\documentclass[11pt]{amsart}

\usepackage{epsfig,overpic,comment}
\usepackage[usenames,dvipsnames,svgnames,table]{xcolor}
\usepackage[hyphens]{url}
\usepackage[pagebackref,linktocpage=true,colorlinks=true,linkcolor=Blue,citecolor=BrickRed,urlcolor=RoyalBlue]{hyperref}
\usepackage[msc-links,abbrev]{amsrefs}
\usepackage{amsmath,amsthm,amssymb,marginnote}
\usepackage{enumerate}
\usepackage[T1]{fontenc}

\usepackage{lmodern}
\usepackage[utf8]{inputenc}

\theoremstyle{definition}

\newcommand{\C}{\mathcal{C}}

\newcommand{\G}{\mathcal{G}}

\newcommand{\RR}{\mathbb{R}}

\newcommand{\HH}{\mathbb{H}}
\newcommand{\SSS}{\mathbb{S}}

\newcommand{\p}{\partial}

\newtheorem{prop}{Proposition}[section]
\newtheorem{thm}[prop]{Theorem}
\newtheorem{lem}[prop]{Lemma}
\newtheorem{cor}[prop]{Corollary}

\theoremstyle{remark}

\numberwithin{equation}{section}
\allowdisplaybreaks
\begin{document}
	\setlength{\baselineskip}{1.2\baselineskip}
	
	\title[Sharp Minkowski-Type Inequality in Cartan-Hadamard 3-Spaces]{Sharp Minkowski-Type Inequality in Cartan-Hadamard 3-Spaces}
	
	\author[F. Hong]{Fang Hong}
	\address{Department of Mathematics and Statistics, McGill University, Montreal, Quebec, H3A 2K6, Canada.}
	\email{\href{mailto:fang.hong@mail.mcgill.ca}{fang.hong@mail.mcgill.ca}}
	\urladdr{\url{https://sites.google.com/view/fang-hong}}

	\begin{abstract}
		In this paper, we prove a sharp Minkowski-type inequality in Cartan-Hadamard 3-spaces using harmonic mean curvature flow and improve the known estimates for total mean curvature in hyperbolic 3-space. In particular, we sharpen Ghomi-Spruck's result in \cite{ghomi-spruck2023} by retaining the volume contribution in the monotonicity argument. As a corollary, we obtain a comparison theorem relating the total mean curvature of convex surfaces in Cartan-Hadamard 3-spaces to their enclosed volume.
	\end{abstract}
	
	\thanks{This research was supported by the Dr. and Mrs. Milton Leong Fellowships in Science and an ISM Graduate Scholarship.}
	
	\maketitle
	%\tableofcontents

	\section{Introduction}
	
	A classical result of Minkowski \cite{minkowski1903} states that,
		for any strictly convex surface $\Gamma$ embedded in Euclidean space $\RR^3$,
		\begin{equation}
			\label{eq:minkowski0}
			M(\Gamma)\geq \sqrt{16\pi S(\Gamma)},
		\end{equation}
		where $S(\Gamma)$ denotes the surface area of $\Gamma$, $M(\Gamma) := \int_\Gamma H\,d\mu$ is the \emph{total mean curvature} of $\Gamma$. Here the \emph{mean curvature} of $\Gamma$ is the trace of the second fundamental form $H:=\textup{trace}(\mathrm{I\!I}_\Gamma)$. Equality holds only when $\Gamma$ is a sphere in $\RR^3$.
	Finding analogues of Minkowski’s inequality \eqref{eq:minkowski0} beyond Euclidean space has been a long-standing problem \cite{santalo1963}, particularly in hyperbolic and Cartan--Hadamard settings.
	In recent years, various Minkowski-type inequalities have been established \cites{gallego-solanes,natario2015,brendle-wang}. Curvature-flow methods have produced several important advances \cites{wang-xia2014,ge-wang-wu2014,andrews-hu-li-2020,scheuer2020,Brendle-Guan-Li}.
	
	Throughout the paper, a \emph{Cartan-Hadamard space}, also called \emph{Cartan-Hadamard manifold} means a complete, simply connected Riemannian space of nonpositive curvature. Cartan-Hadamard spaces include Euclidean and hyperbolic spaces as basic examples. The starting point for the present work is the following theorem of Ghomi and Spruck \cite{ghomi-spruck2023}. They proved that for any smooth convex surface $\Gamma$ in a Cartan-Hadamard 3-space $N$ with curvature $K\leq a\leq 0$,
	\begin{equation}\label{eq:minkowski2}
		M(\Gamma)\geq \sqrt{16\pi S(\Gamma)- 2 a S(\Gamma)^2},
	\end{equation}
	where equality holds only if the domain enclosed by $\Gamma$ is isometric to a ball in $\RR^3$.
	
	Consider the hyperbolic space $\HH^3(a)$ with constant curvature $a \leq 0$. We define the \emph{isoperimetric profile function} $\eta_{0,a}(x)$ and the \emph{total mean curvature profile function} $\xi_{0,a}(x)$ of $\HH^3(a)$. They are functions $[0, \infty) \to [0,\infty)$ such that $\eta_{0,a}(x)$ is the surface area of the (geodesic) sphere in $\HH^3(a)$ with volume $x$, and $\xi_{0,a}(x)$ is the total mean curvature of the sphere in $\HH^3(a)$ with volume $x$. Our contribution is to modify the harmonic mean curvature flow argument so that it keeps track of the enclosed volume, leading to the strengthened estimate below.
	
	\begin{thm}\label{thm:minkowski1}
		Let $\Omega$ be a domain in a Cartan-Hadamard 3-space $N$ with curvature $K\leq a\leq 0$, such that its boundary $\Gamma$ is a smooth strictly convex surface. Then 
		\begin{equation}
			\label{eq:minkowski1}
			M(\Gamma)\geq \sqrt{16\pi S(\Gamma)- 2 a S(\Gamma)^2 - 2 a \eta_{0,a}(V(\Omega))^2},
		\end{equation}
		where $V(\Omega)$ denotes the volume of $\Omega$, and $\eta_{0,a}$ is the isoperimetric profile function of $\HH^3(a)$.
		Equality holds only if $\Omega$ is isometric to a ball in $\HH^3(a)$.
	\end{thm}
	
	Theorem \ref{thm:minkowski1} immediately gives a lower bound of total mean curvature depending only on the enclosed volume.
	
	\begin{cor}\label{cor:minkowski1}
		Let $\Omega$ be a domain in a Cartan-Hadamard 3-space $N$ with curvature $K\leq a\leq 0$, such that its boundary $\Gamma$ is a smooth strictly convex surface. Then 
		\begin{equation}
			\label{eq:cor-minkowski1}
			M(\Gamma)\geq \xi_{0,a}(V(\Omega)),
		\end{equation}
		where $\xi_{0,a}$ is the total mean curvature profile function of  $\HH^3(a)$.
		Equality holds only if $\Omega$ is isometric to a ball in $\HH^3(a)$.
	\end{cor}
	
	Inequality \eqref{eq:minkowski1} is a refinement of \eqref{eq:minkowski2}. Compared with the previously known inequalities, \eqref{eq:minkowski1} appears to give the sharpest lower bound of total mean curvature currently known in Cartan-Hadamard 3-spaces, including hyperbolic 3-space.

	Total mean curvature enters naturally both in convex geometry and in variational problems for area. It is the main term of one of the quermassintegrals for convex bodies in space forms, which are important in convex geometry and Brunn-Minkowski Theory \cite{schneider2014}. It also plays an important role in the definition of Brown-York quasi-local mass in general relativity \cite{brown-york1993}. A longstanding problem related to the Minkowski inequality \eqref{eq:minkowski0} is the still-open question of its validity for general mean-convex domains. 
	Guan-Li \cite[Theorem 2]{guan-li2009} proved that \eqref{eq:minkowski0} holds provided that $\Gamma$ is star-shaped and mean-convex. G. Huisken showed that \eqref{eq:minkowski0} holds for outward-minimizing surfaces, see \cite[Theorem 6]{guan-li2009} and also \cite{huisken2009}. Dalphin-Henrot-Masnou-Takahashi \cite[Theorem 1.1]{dalphin2016} established \eqref{eq:minkowski0} in the case where $\Gamma$ is axially symmetric and such that $\Gamma \cap P$ is connected for every affine plane $P$ orthogonal to the axis of symmetry.
	
	Non-sharp inequalities involving the two quantities, total mean curvature $M$ and surface area $S$, have been studied. For example, $ M(\Gamma)\geq \sqrt{-a}\,S(\Gamma)$ by Gallego-Solanes \cites{gallego-solanes} in $\HH^3(a)$ (note that in \cite{gallego-solanes}, $H:=\textup{trace}(\mathrm{I\!I}_\Gamma)/(n-1)$). 
	In this paper, we establish a sharp result involving three quantities: total mean curvature $M$, surface area $S$, and the volume $V$ enclosed by the surface for convex surfaces in Cartan-Hadamard 3-spaces. 
	
	A sharp inequality between total mean curvature $M$ and surface area $S$ remains unknown. Santal\'{o} conjectured \cite{santalo1963} that (see  \cite[p. 78]{santalo2009}), in hyperbolic space $\HH^3(a)$ with constant curvature $a \leq 0$ we have
	\begin{equation}\label{eq:santalo-conj1}
	M(\Gamma) \geq \sqrt{16\pi S(\Gamma)-4a S(\Gamma)^2}.  
	\end{equation}
	The proposed right-hand side of \eqref{eq:santalo-conj1} is exactly the value attained by a geodesic sphere with area \(S(\Gamma)\). Thus, if \eqref{eq:santalo-conj1} holds, geodesic spheres minimize total mean curvature among convex surfaces with fixed area. However, an example by Naveira-Solanes  \cite[p. 815]{santalo2009}, see  \cite[p. 109]{natario2015} or \cite[Note 1.3]{ghomi-spruck2023}, shows that \eqref{eq:santalo-conj1} is false in general. They showed that a flat double disk, which is isometric to a geodesic sphere in a totally geodesic plane $\HH^2$ embedded in $\HH^3$, gives a counterexample to \eqref{eq:santalo-conj1} when its surface area $S(\Gamma)$ is large enough. Its two flat faces are counted in the surface area and its circular edge is counted in the singular total mean curvature.
	
	In $\mathbb H^3$, for a given area, the minimizer of the total mean curvature among convex surfaces with fixed surface area exists by Blaschke selection theorem. The optimal horo-convex minimizer of total mean curvature $M$ with fixed surface area $S$ has been proved to be the geodesic sphere \cite[Theorem 6.1]{ge-wang-wu2014}. Yet, the shape of the general convex minimizer is not known. Hence \emph{Santal\'{o}'s problem}, which asks for the optimal convex surface with the minimum total mean curvature $M$ among convex surfaces with fixed surface area $S$, is still open. Recently the author \cite{Hong2026} showed that there exists a new family of counterexamples to \eqref{eq:santalo-conj1} in addition to flat double disks. Santal\'{o}'s problem is of interest not only in geometry but also in general relativity. For example, the total mean curvature in hyperbolic space is also used in the definition of Wang-Yau's quasi-local mass in \cite[Theorem 1.3]{wang-yau2007}.
	
	A natural counterpart of total mean curvature in hyperbolic space $\HH^{n+1}$ is the  \emph{quermassintegral} $A_1$. For any compact convex set $\Omega$ in $\HH^{n+1}$, $A_1(\Omega)$ is defined by
	\begin{align}\label{defn:A_1}
		A_1(\Omega) := M(\Gamma) - n V(\Omega) 
	\end{align}
	where $\Gamma$ is the boundary of $\Omega$, and $V(\Omega)$ denotes the volume of $\Omega$. In \cite{Brendle-Guan-Li}, Brendle, Guan and Li proved that, for any domain $\Omega$ with smooth mean convex boundary $\Gamma$ in $\HH^3$, 
	\begin{equation}\label{eq:minkowski3-A1}
		A_1(\Omega) \geq \sqrt{S(\Gamma)}\sqrt{S(\Gamma)+4\pi} + 4\pi \operatorname{arcsinh}\left( \sqrt{\frac{S(\Gamma)}{4\pi}} \right)
		,
	\end{equation}
	where we say a hypersurface $\Gamma$ is \emph{mean convex} if the mean curvature $H$ is non-negative on $\Gamma$. Equality holds only when $\Gamma$ is a geodesic sphere in $\HH^3$. By \eqref{defn:A_1}, the inequality \eqref{eq:minkowski3-A1} can be formulated as
	\begin{equation}\label{eq:minkowski3}
		M(\Gamma) \geq \sqrt{S(\Gamma)}\sqrt{S(\Gamma)+4\pi} + 4\pi \operatorname{arcsinh}\left( \sqrt{\frac{S(\Gamma)}{4\pi}} \right)
		+ 2 V(\Omega).
	\end{equation}
	 
	%Inequality \eqref{eq:minkowski1} is sharper than inequality \eqref{eq:minkowski3}  when the isoperimetric profile function $\eta_{0,a}(V)$ is sufficient small comparing to the surface area.
	In fact, inequality \eqref{eq:minkowski1} proved in this paper is also sharper than \eqref{eq:minkowski3} for any convex surface $\Gamma$ in $\HH^3$, unless it is a geodesic sphere, see Section \ref{sec:4}.
	
	\section{Preliminaries}\label{sec:1}
	
		\subsection{Notation and Definitions of Convexity}
		Here we list some notation and definitions for convexity used in this paper. By \emph{smooth} we mean $\C^\infty$. The term \emph{curvature} means sectional curvature unless specified otherwise. A \emph{domain} is a connected open set with compact closure.
		Assuming smoothness, a \emph{convex} hypersurface $\Gamma$ of an ambient manifold $N$ is a closed embedded submanifold of codimension one which, when properly oriented, has positive semidefinite second fundamental form $\mathrm{I\!I}_\Gamma$. A \emph{strictly convex} hypersurface $\Gamma$ of $N$ is a convex hypersurface with positive second fundamental form. %A \emph{mean convex} hypersurface is a closed embedded submanifold of codimension one which has non-negative mean curvature. 
	
		\subsection{Notation and Facts about Profile Functions}
		Here we list some facts about profile functions in hyperbolic space $\HH^3(a)$ with constant curvature $a \leq 0$. For a geodesic sphere $\SSS(r)$ with radius $r$ in $\HH^3(a)$, we denote its enclosed volume, area, and total mean curvature as functions of $r$ by $V_{a}^{B}(r) $, $S_{a}^{B}(r)$ and $M_{a}^{B}(r)$, respectively. 
		
		It is well known that
		\begin{align}\label{note1.defn.VSM}
			V(\SSS(r)) &= 2\pi \frac{1}{\left(\sqrt{-a}\right)^3}
			\left( \sinh(\sqrt{-a}r) \cosh(\sqrt{-a}r) - \sqrt{-a}r \right) =: V_{a}^{B}(r),
			\\
			\nonumber
			S(\SSS(r)) &= 4\pi \frac{1}{\left(\sqrt{-a}\right)^2} \sinh(\sqrt{-a}r)^2 =: S_{a}^{B}(r),
			\\
			\nonumber
			M(\SSS(r)) &= 8\pi \frac{1}{\sqrt{-a}} \sinh(\sqrt{-a}r) \cosh(\sqrt{-a}r) =: M_{a}^{B}(r),
		\end{align}
		and in particular,
		\begin{align}\label{note1.eq.derivative}
			(V_{a}^{B})'(r) &= S_{a}^{B}(r), \quad
			(S_{a}^{B})'(r) = M_{a}^{B}(r).
		\end{align}
		
		By the definition of the isoperimetric profile function $\eta_{0,a}(x)$ and the total mean curvature profile function $\xi_{0,a}(x)$, we have for any $r\geq0$,
		\begin{align}\label{note1.defn.SM}
			\eta_{0,a}(V_{a}^{B}(r)) &= S_{a}^{B}(r),\\
			\nonumber
			\xi_{0,a}(V_{a}^{B}(r)) &= M_{a}^{B}(r).
		\end{align}

		From \eqref{note1.defn.VSM}, we have for any $x\geq 0$,
		\begin{align}\label{SM.16pi-4a}
			\xi_{0,a}(x) = \sqrt{ 16\pi \eta_{0,a}(x) - 4 a (\eta_{0,a}(x))^2 }.
		\end{align}
		
		We also have the following lemma.
		\begin{lem}
			For any $a \leq 0$, and any $x\geq0$, we have
			\begin{equation}\label{SM.lemma}
				\eta_{0,a}'(x) \eta_{0,a}(x) = \xi_{0,a}(x),
			\end{equation}
			where the functions $\eta_{0,a}$ and $\xi_{0,a}$ are as defined in \eqref{note1.defn.SM}.
		\end{lem}
		\begin{proof}
			By definition, we have for any $r>0$,
			$\eta_{0,a}\left( V_{a}^{B}(r) \right) = S_{a}^{B}(r)$.
			Differentiating both sides with respect to $r$, we have
			\begin{equation}\label{SM.lemma.temp.2}
				\eta_{0,a}'\left( V_{a}^{B}(r) \right) (V_{a}^B)'(r) = (S_{a}^{B})'(r).
			\end{equation}
			By \eqref{note1.eq.derivative}, we obtain from \eqref{SM.lemma.temp.2} that
			$
			\eta_{0,a}'\left( V_{a}^{B}(r) \right) S_{a}^{B}(r) = M_{a}^{B}(r),
			$
			that is,
			$$
			\eta_{0,a}'\left( V_{a}^{B}(r) \right) \eta_{0,a}\left( V_{a}^{B}(r) \right) = \xi_{0,a}\left( V_{a}^{B}(r) \right),
			$$
			that is, for any $x>0$,
			$
			\eta_{0,a}'(x) \eta_{0,a}(x) = \xi_{0,a}(x).
			$
		\end{proof}
		
		The isoperimetric profile function $\eta_{0,a}$ plays an important role in geometric inequalities in Cartan-Hadamard 3-spaces. We state the following isoperimetric inequality due to Kleiner \cite{kleiner1992}.
		\begin{thm}[B. Kleiner, 1992]\label{thm:isoperi1}
			Let $\Omega$ be a domain in a Cartan-Hadamard 3-space $N$ with curvature $K\leq a\leq 0$, such that its boundary $\Gamma$ is a smooth surface. Then
			\begin{equation}
				\label{eq:isoperi1}
				S(\Gamma) \geq \eta_{0,a}(V(\Omega)).
			\end{equation}
			Equality holds only if $\Omega$ is isometric to a ball in $\HH^3(a)$.
		\end{thm}

		\subsection{Notation and Facts about Geometric Flows}
		Here we list some evolution equations along geometric flows.
		A \emph{geometric flow}  of a hypersurface $\Gamma$ in a Riemannian $(n+1)$-manifold $N$ \cite{andrews-chow2020,giga2006,huisken-polden1999} is a one-parameter family of immersions $X\colon\Gamma\times[0,T)\to N$, $X_t(\cdot):=X(\cdot, t)$,  given by
		\begin{equation}\label{eq:hmcf}
			X_t'(p)=-F_t(p)\nu_t(p),\quad\quad\quad X_0(p)=p,   
		\end{equation}
		where $(\cdot)':=\partial/\partial t(\cdot)$,
		$\nu_t$ is a normal vector field along $\Gamma_t:=X_t(\Gamma)$, and the \emph{speed function} $F_t$ depends on \emph{principal curvatures}, that is, eigenvalues $\kappa_i^t$  of the second fundamental form $\mathrm{I\!I}_t:=\mathrm{I\!I}_{\Gamma_t}$. %More precisely, $\nu_t(p)$ is the normal vevtor and $\kappa_i^t(p)$ are the principal curvatures of $\Gamma_t$ at the point $X_t(p)$. 
		
		Let $d\mu_t$ be the area element induced on $\Gamma$ by $X_t$. $G_t:=\det(\mathrm{I\!I}_t)$ and $H_t:=\textup{trace}(\mathrm{I\!I}_t)$ are the \emph{Gauss-Kronecker curvature} and \emph{mean curvature} of $\Gamma_t$, respectively. Let $\Omega_t$ be the domain enclosed by $\Gamma_t$ in $N$. Let $V(\Omega_t)$ be the volume of $\Omega_t$ in $N$.
		By \cite[Thm. 3.2(v)]{huisken-polden1999} and \cite[Lem. 7.4]{huisken-polden1999} (see also for example \cite{andrews1994}, \cite{ghomi-spruck2023}), for any geometric flow \eqref{eq:hmcf},
			\begin{align}
				\label{eq:evolution}
			\frac{d}{dt}(H_t) &= 
			\Delta_t F_t+\left(|\mathrm{I\!I}_t|^2+\operatorname{Ric}(\nu_t)\right)F_t,
			\\
			\nonumber
			\frac{d}{dt}(d\mu_t) &= -F_tH_td\mu_t,
			\\
			\nonumber
			\frac{d}{dt}V(\Omega_t) &= -\int_{\Gamma} F_t d\mu_t,
			\\
			\nonumber
			\frac{d}{dt} S(\Gamma_t) &= -\int_{\Gamma} F_t H_t d\mu_t,
		\end{align}
		where $|\mathrm{I\!I}_t|:=\sqrt{\sum(\kappa_i^t)^2}$, $\Delta_t$ is the Laplace-Beltrami operator induced on $\Gamma$ by $X_t$, and $\operatorname{Ric}(\nu_t)$ is the Ricci curvature of $N$ at the point $X_t(p)$ in the direction of $\nu_t(p)$, that is, the sum of sectional curvatures of $N$ with respect to a pair of orthogonal planes that contain $\nu_t(p)$.
		
		%Let $H$ be the function on $\Omega\setminus\{o\}$ given by $H(X_t(p)):=H_t(p)$. Also 
		%define $u$ on $\Omega\setminus\{o\}$ by  $u(X_t(p))=t$, which yields that $|\nabla u(X_t)|=1/F_t$. Then $H=\textup{div}(\nabla u/|\nabla u|)$, and Stokes' theorem together with the coarea formula yields that
		%$$
		%S(\Gamma_t)-S(\Gamma_{t+h})=\int_{\Omega_t\setminus\Omega_{t+h}}H
		%=
		%\int_t^{t+h}\left(\int_{\Gamma}H_sF_s\,d\mu_s\right)ds
		%$$
		%where $\Omega_t$ is the convex domain bounded by $\Gamma_t$. Hence
		%\begin{align}\label{eq:evolution-area}
		%	\frac{d}{dt} S(\Gamma_t) &= -\int_{\Gamma} F_t H_t d\mu_t.
		%\end{align}

\section{Proof of Theorem \ref{thm:minkowski1} and Corollary \ref{cor:minkowski1}}\label{sec:2}

Following \cite{ghomi-spruck2023}, we will prove Theorem \ref{thm:minkowski1} via harmonic mean curvature flow. If, in a geometric flow \eqref{eq:hmcf}, the speed \(F_t\) is the harmonic mean of the \(\kappa_i^t\), that is,
$$
F_t=\left(\sum_{i=1}^{n}\frac{1}{\kappa_i^t}\right)^{-1},
$$
then $X$ is called the \emph{harmonic mean curvature flow} of $\Gamma$. In particular, when $n=2$,
$$
F_t=\frac{G_t}{H_t}.
$$

Xu showed \cite[Thm. 1.2]{xu2010, gulliver-xu2009} that, if $\Gamma$ is a smooth strictly convex hypersurface in a Cartan-Hadamard space $N$ and $F_t$ is the harmonic mean curvature, $X$ exists for $t\in [0,T)$, is $\C^\infty$, and $\Gamma_t$ are strictly convex hypersurfaces converging to a point as $t\to T$. Ghomi-Spruck \cite{ghomi-spruck2023} emphasized that this makes harmonic mean curvature flow the natural contracting flow for this setting. In Cartan-Hadamard spaces, this is the only geometric flow known to preserve the convexity of a hypersurface in $N$ while contracting it to a point. 

	Fix a smooth strictly convex initial surface \(\Gamma\) in a Cartan-Hadamard 3-space $N$ with curvature $K\leq a\leq 0$. Let \(\Gamma_t\) denote its evolution surface along harmonic mean curvature flow, defined up to the extinction time \(T\). Thus, $\Gamma_t$ converges to a point $o$ in $N$ as $t\to T$. Set $M_t:=M(\Gamma_t)$. The key idea of Ghomi-Spruck's proof of \eqref{eq:minkowski2} in \cite{ghomi-spruck2023} is the following monotonicity:
	\begin{prop}
		Along harmonic mean curvature flow, the function
		\begin{align}\label{phi:ghomi-spruck}
			\phi(t) := M_t^2-16\pi S(\Gamma_t) +2a S(\Gamma_t)^2
		\end{align}
		is monotonically non-increasing.
	\end{prop}
	In Ghomi-Spruck's proof \cite{ghomi-spruck2023}, the volume term appearing in the evolution equation of $\phi(t)$ was neglected, and therefore the inequality is non-sharp in Cartan-Hadamard spaces with strictly negative curvature. The equality case would force the domain enclosed by $\Gamma$ to be isometric to a subset of $\RR^3$, and hence equality cannot hold in general non-Euclidean Cartan-Hadamard 3-spaces. 
	
	In this paper, we will take the volume term into consideration and refine the inequality \eqref{eq:minkowski2}. However, due to the lack of a sharp inequality comparing total mean curvature $M$ and volume $V$ in Cartan-Hadamard 3-spaces (which is Corollary \ref{cor:minkowski1}, proved below), we cannot prove \eqref{eq:minkowski1} directly by setting a single auxiliary function and proving its monotonicity along the flow. We therefore bootstrap the desired estimate through an iterative construction of auxiliary functions and proving their monotonicity along the flow. %In fact, we will use a corollary of \eqref{eq:minkowski2} comparing total mean curvature and volume in the first step of our iteration.

\begin{proof}[Proof of Theorem \ref{thm:minkowski1} and Corollary \ref{cor:minkowski1}]
	\hfill\\
	\textbf{Step 1: The first step of the iteration}
	
	Given any smooth strictly convex surface $\Gamma$ in a Cartan-Hadamard 3-space $N$ with curvature $K\leq a\leq 0$, we let $\Gamma_t$, $t\in[0,T)$, be the surfaces generated by the harmonic mean curvature flow of $\Gamma$, converging to a point $o$ in $N$. Let $\Omega_t$ be the domain enclosed by $\Gamma_t$ in $N$.
	Set $M_t:=M(\Gamma_t)$, and
	$$
	\phi_1(t):= M_t^2-16\pi S(\Gamma_t) +2a S(\Gamma_t)^2 -  P_1(V(\Omega_t)) ,
	$$
	where 
	\begin{equation}\label{defn:F_1}
		P_1(x):= -4a \int_{0}^{x} Q_1(t) dt.
	\end{equation}
	Here $Q_1(x) := \sqrt{16\pi \eta_{0,a}(x) - 2 a \eta_{0,a}(x)^2}$.
	It remains to prove non-negativity of \(\phi_1\) at the initial time. We start by differentiating it along the flow. 
	
	By \eqref{eq:evolution}, for the harmonic mean curvature speed \(F_t=G_t/H_t\), the area element evolves by $\frac{d}{dt}(d\mu_t) = -F_tH_td\mu_t = - G_t d\mu_t $. The curvature upper bound of $N$ gives \(\operatorname{Ric}(\nu_t)\le 2a\). Substitution into the evolution equation \eqref{eq:evolution} yields
	\begin{eqnarray}\label{eq:G'}\notag
		\frac{d}{dt}(M_t)&=&\int_{\Gamma}\left(\frac{d}{dt}(H_t)d\mu_t+H_t\frac{d}{dt}(d\mu_t)\right)\\ 
		&=&
		\int_{\Gamma}\Big(\Delta_t F_t+\big(|\mathrm{I\!I}_t|^2-(H_t)^2\big)F_t +\operatorname{Ric}(\nu_t)F_t\Big)d\mu_t\\ \notag
		&=&
		-2\int_{\Gamma}\frac{(G_t)^2}{H_t}d\mu_t
		+ \int_{\Gamma} \operatorname{Ric}(\nu_t) F_t d\mu_t \\ \notag
		&\leq&
		-2\int_{\Gamma}\frac{(G_t)^2}{H_t}\,d\mu_t + 2a \int_{\Gamma} F_t d\mu_t \\
		\notag
		&= &
		-2\int_{\Gamma}\frac{(G_t)^2}{H_t}\,d\mu_t - 2a \left(\frac{d}{dt}V(\Omega_t)\right).
	\end{eqnarray}
	Applying Cauchy--Schwarz to the curvature integrals gives
	\begin{align}\label{eq:CS}
		M_t \left( \frac{d}{dt} M_t \right)
		\leq
		-2 M_t\int_{\Gamma}\frac{(G_t)^2}{H_t}d\mu_t - 2a M_t \frac{d}{dt}V(\Omega_t)
		\leq
		-2\G_t^2 -2a M_t \frac{d}{dt}V(\Omega_t),
	\end{align}
	where $\G_t=\G(\Gamma_t):=\int_{\Gamma}G_t d\mu_t$ is the \emph{total Gauss-Kronecker curvature} of $\Gamma_t$.
	By \eqref{eq:evolution}, we also have
	$$
	\frac{d}{dt}S(\Gamma_t)=-\G_t.
	$$
	Thus, from the definition of $P_1$ as in \eqref{defn:F_1}, we have
	\begin{align}
		\label{eq:phi'}
		&\frac{d}{dt}\phi_1(t) 
		\\
		\nonumber
		=&
		2 M_t \frac{d}{dt}M_t
		-16\pi \frac{d}{dt}S(\Gamma_t) 
		+4aS(\Gamma_t)\frac{d}{dt}S(\Gamma_t)
		-  P_1'(V(\Omega_t)) \frac{d}{dt} V(\Omega_t)
		\\
		\nonumber
		\leq
		&
		-4 \G_t^2 -4 a M_t  \frac{d}{dt} V(\Omega_t) 
		-16\pi \frac{d}{dt}S(\Gamma_t) 
		+4aS(\Gamma_t)\frac{d}{dt}S(\Gamma_t)
		-  P_1'(V(\Omega_t)) \frac{d}{dt} V(\Omega_t)
		\\
		\nonumber
		=&
		-4 \G_t ( \G_t - 4\pi + a S(\Gamma_t))
		+
		(-4a M_t - P_1'(V(\Omega_t)) )\frac{d}{dt} V(\Omega_t)
		\\
		\nonumber
		=&
		-4 \G_t ( \G_t - 4\pi + a S(\Gamma_t))
		+
		(-4a)( M_t - Q_1(V(\Omega_t)) )\frac{d}{dt} V(\Omega_t).
	\end{align}
	
	By Gauss' equation, for all $p\in\Gamma_t$, 
	\begin{equation}\label{eq:gauss}
	G_t(p)=K_{\Gamma_t}(p)-K_N(T_p\Gamma_t),
	\end{equation}
	where $K_{\Gamma_t}$ is the sectional curvature of $\Gamma_t$, and $K_N(T_p\Gamma_t)$ is the sectional curvature of $N$ with respect to the tangent plane $T_p \Gamma_t\subset T_p N$.   Thus, by Gauss-Bonnet theorem, 
	\begin{equation}\label{eq:GB}
	\G_t=4\pi -\int_{p\in\Gamma_t} K_N(T_p\Gamma_t)\,d\mu_t\geq 4\pi-aS(\Gamma_t) .
	\end{equation}
	From \eqref{eq:minkowski2} and Theorem \ref{thm:isoperi1}, we have for any $t$, 
	\begin{equation}\label{tmp: MgeqQ1}
		M_t \geq \sqrt{16\pi S(\Gamma_t)- 2 a S(\Gamma_t)^2} \geq 
		\sqrt{16\pi \eta_{0,a}(V(\Omega_t))- 2 a \eta_{0,a}(V(\Omega_t))^2}
		= Q_1(V(\Omega_t)).
	\end{equation}
	
	Combining \eqref{eq:GB} and \eqref{tmp: MgeqQ1} and substituting them into \eqref{eq:phi'}, we obtain $\frac{d}{dt} \phi_1(t)\leq 0$ as claimed. At the extinction time, the enclosed bodies shrink to a point; in particular \(S(\Gamma_t)\to0\) and \(V(\Omega_t)\to0\). Hence
	$$
	\lim_{t\to T}\phi_1(t) \geq \liminf_{t\to T} M_t^2\geq 0.
	$$ 
	Thus $\phi_1(0)\geq 0$, which yields the first inequality in our iteration: For any smooth strictly convex $\Gamma$ in $N$, with enclosed domain $\Omega$,
	\begin{equation}\label{eq:iter1}
		M(\Gamma) \geq \sqrt{16\pi S(\Gamma) -2 a S(\Gamma)^2 + P_1(V(\Omega))}.
	\end{equation}
	
	\textbf{Step 2: General iteration}
	
	As a corollary of \eqref{eq:iter1}, by Theorem \ref{thm:isoperi1}, we have a new inequality between total mean curvature and volume. For any smooth strictly convex surface $\Gamma$ in $N$, with enclosed domain $\Omega$, we obtain
	\begin{equation}\label{tmp: MgeqQ2}
		M(\Gamma) \geq \sqrt{16\pi \eta_{0,a}(V(\Omega)) -2 a \eta_{0,a}(V(\Omega))^2 + P_1(V(\Omega))}.
	\end{equation} 
	Comparing \eqref{tmp: MgeqQ1} and \eqref{tmp: MgeqQ2}, we have refined the original inequality between total mean curvature and volume in Cartan-Hadamard 3-spaces. We may use \eqref{tmp: MgeqQ2} to construct a new auxiliary function, prove its monotonicity, and repeat the process inductively. 
	
	To state the general iteration step, we define a sequence of functions $\{Q_n\}_{n=1}^{\infty}$ on $[0,\infty)$, as follows: 
	$Q_1(x) := \sqrt{16 \pi\eta_{0,a}(x) - 2a \eta_{0,a}(x)^2}$, and for any positive integer $n$,
	\begin{equation}\label{defn:Q_n}
		Q_{n+1}(x) := \sqrt{16 \pi \eta_{0,a}(x) - 2a \eta_{0,a}(x)^2 - 4a \int_{0}^{x} Q_n(t)dt }.
	\end{equation}
	
	For any positive integer $n$, we define
	\begin{equation*}
		\phi_n(t) := M_t^2 - 16\pi S(\Gamma_t) + 2 a S(\Gamma_t)^2 - P_n(V(\Omega_t)),
	\end{equation*}
	where $P_n(x):= -4a \int_{0}^{x} Q_n(t) dt$. 
	Assume that for any smooth strictly convex surface $\Gamma$ in $N$, we have $M(\Gamma) \geq Q_{n}(V(\Omega))$. We will prove that for any smooth strictly convex surface $\Gamma$ in $N$, $ \frac{d}{dt} \phi_n(t) \leq 0$, and $M(\Gamma) \geq Q_{n+1}(V(\Omega))$ will follow as a corollary.
	
	As in \eqref{eq:phi'}, we compute 
	\begin{align}
		\label{eq:iterate:phi'}
		&\frac{d}{dt}\phi_n(t) 
		\\
		\nonumber
		=&
		2 M_t \frac{d}{dt}M_t
		-16\pi \frac{d}{dt}S(\Gamma_t) 
		+4aS(\Gamma_t)\frac{d}{dt}S(\Gamma_t)
		-  P_n'(V(\Omega_t)) \frac{d}{dt} V(\Omega_t)
		\\
		\nonumber
		\leq
		&
		-4 \G_t^2 -4 a M_t  \frac{d}{dt} V(\Omega_t) 
		-16\pi \frac{d}{dt}S(\Gamma_t) 
		+4aS(\Gamma_t)\frac{d}{dt}S(\Gamma_t)
		-  P_n'(V(\Omega_t)) \frac{d}{dt} V(\Omega_t)
		\\
		\nonumber
		=&
		-4 \G_t ( \G_t - 4\pi + a S(\Gamma_t))
		+
		(-4a M_t - P_n'(V(\Omega_t)) )\frac{d}{dt} V(\Omega_t)
		\\
		\nonumber
		=&
		-4 \G_t ( \G_t - 4\pi + a S(\Gamma_t))
		+
		(-4a)( M_t - Q_n(V(\Omega_t)) )\frac{d}{dt} V(\Omega_t).
	\end{align}
	By \eqref{eq:GB} and the assumption that $M(\Gamma) \geq Q_{n}(V(\Omega))$ for any smooth strictly convex surface $\Gamma$ with enclosed domain $\Omega$, we have 
	\begin{align*}
		\frac{d}{dt}\phi_n(t) \leq 0,
	\end{align*}
	and hence $\phi_n(0) \geq 0$, that is,
	\begin{equation*}
		M(\Gamma) \geq \sqrt{16\pi S(\Gamma) -2 a S(\Gamma)^2 + P_n(V(\Omega))}.
	\end{equation*}
	By Theorem \ref{thm:isoperi1}, as $Q_{n+1}$ is defined in \eqref{defn:Q_n}, we have
	\begin{equation*}
		M(\Gamma) \geq \sqrt{16\pi \eta_{0,a}(V(\Omega)) -2 a \eta_{0,a}(V(\Omega))^2 + P_n(V(\Omega))} = Q_{n+1}(V(\Omega)),
	\end{equation*}
	which completes the induction.

	In general, for any positive integer $n$, and for any strictly convex surface $\Gamma$ in $N$ with enclosed domain $\Omega$, we have
	\begin{equation}\label{tmp:minkowski_n}
		M(\Gamma) \geq \sqrt{16 \pi S(\Gamma) - 2a S(\Gamma)^2 - 4a \int_{0}^{V(\Omega)} Q_n(t)dt }.
	\end{equation}
	
	\textbf{Step 3: On the limit of the sequence of functions $\{Q_n\}$}
	
	Clearly for any positive integer $n$, $Q_n$ is continuous on $[0,\infty)$, $C^{\infty}$ on $(0,\infty)$ with $Q_n(0)=0$, and monotonically non-decreasing.
	Since $Q_{n+1}(x) \geq Q_{n}(x)$ implies
	\begin{align*}
		Q_{n+2}(x) = &\sqrt{16 \pi \eta_{0,a}(x) - 2a \eta_{0,a}(x)^2 - 4a \int_{0}^{x} Q_{n+1}(t)dt} \\
		\geq 
		&\sqrt{16 \pi \eta_{0,a}(x) - 2a \eta_{0,a}(x)^2 - 4a \int_{0}^{x} Q_{n}(t)dt} = Q_{n+1}(x),
	\end{align*}
	we have for any $x\in (0,\infty)$, $ Q_{n+1}(x) \geq Q_{n}(x) $ by induction.
	For any $x>0$, \eqref{tmp:minkowski_n} implies the sequence $\{Q_{n}(x)\}_{n\geq 1}$ is bounded. Hence there exists a limit function $Q_{\infty}: [0,\infty) \rightarrow [0,\infty)$ such that $\lim_{n \rightarrow \infty} Q_n = Q_{\infty}$ pointwise on $[0,\infty)$ and uniformly on any compact subset of $[0,\infty)$. Moreover, for any $x>0$, $Q_{\infty}$ is differentiable at $x$. By taking the limit in \eqref{tmp:minkowski_n}, we have
	\begin{equation}\label{tmp:minkowski_infty}
		M(\Gamma) \geq \sqrt{16 \pi S(\Gamma) - 2a S(\Gamma)^2 - 4a \int_{0}^{V(\Omega)} Q_{\infty}(t)dt }.
	\end{equation}
	
	We claim that $Q_{\infty} = \xi_{0,a}$, where $\xi_{0,a}$ is the total mean curvature profile function in $\HH^3(a)$ as defined in \eqref{note1.defn.SM}. To prove this, we  take the limit on both sides of \eqref{defn:Q_n} to obtain an integral equation satisfied by $Q_{\infty}$.
	\begin{equation}\label{ODE.tmp1}
		Q_{\infty}(x) = \sqrt{16 \pi \eta_{0,a}(x) - 2a \eta_{0,a}(x)^2 - 4a \int_{0}^{x} Q_{\infty}(t)dt }.
	\end{equation}
	Combining \eqref{SM.16pi-4a} and \eqref{SM.lemma}, we have
	$16 \pi \eta_{0,a}(x) - 2a \eta_{0,a}(x)^2 = \xi_{0,a}(x)^2 + 4a \int_{0}^{x} \xi_{0,a}(t)dt  $. Substituting this into \eqref{ODE.tmp1} and simplifying, we get 
	\begin{equation}\label{ODE.tmp2}
		Q_{\infty}(x)^2 + 4a \int_{0}^{x} Q_{\infty}(t)dt
		=
		\xi_{0,a}(x)^2 + 4a \int_{0}^{x} \xi_{0,a}(t) dt.
	\end{equation}
	
	Note that \eqref{tmp:minkowski_infty} holds for any Cartan-Hadamard 3-space $N$ with curvature $K\leq a\leq0$, and any strictly convex $\Gamma$ in $N$. If we take $N$ to be $\HH^3(a)$ and $\Gamma$ to be a geodesic sphere $\SSS^2(r)$ of radius $r$ in $\HH^3(a)$, by \eqref{SM.16pi-4a} and \eqref{SM.lemma}, we have
	\begin{equation}\label{tmp.tmp.sphere}
		M(\Gamma) = \sqrt{16\pi S(\Gamma) - 4a S(\Gamma)^2}
		= \sqrt{16\pi S(\Gamma) - 2a S(\Gamma)^2 - 4a \int_{0}^{V(\Omega)}\xi_{0,a}(t)dt}
	\end{equation}
	Comparing \eqref{tmp.tmp.sphere} and \eqref{tmp:minkowski_infty}, we have for any $r\geq0$,
	$$
	\int_{0}^{V(\SSS^2(r))}\xi_{0,a}(t)dt
	\geq
	\int_{0}^{V(\SSS^2(r))}Q_{\infty}(t)dt,
	$$
	that is, for any $x\geq0$,
	\begin{equation}\label{tmp.tmp.compare1}
		\int_{0}^{x} \xi_{0,a}(t)dt
		\geq
		\int_{0}^{x} Q_{\infty}(t)dt.
	\end{equation}
	Comparing \eqref{tmp.tmp.compare1} with \eqref{ODE.tmp2}, we have for any $x\geq0$, 
	\begin{equation}\label{tmp.compare11}
		\xi_{0,a}(x) \geq Q_{\infty}(x).
	\end{equation}
	
	Note that by \eqref{ODE.tmp1} and \eqref{ODE.tmp2}, 
	\begin{equation*}
		Q_{\infty}(x)^2 + 4a \int_{0}^{x} Q_{\infty}(t)dt
		=
		\xi_{0,a}(x)^2 + 4a \int_{0}^{x} \xi_{0,a}(t) dt
		=
		16\pi \eta_{0,a}(x) -2a \eta_{0,a}(x)^2,
	\end{equation*}
	is a strictly increasing function of $x$. Hence, differentiating both sides of \eqref{ODE.tmp2}, we have for any $x>0$,
	\begin{equation}
		\label{ODE.tmp3}
		Q_{\infty}(x) \left( Q_{\infty}'(x) + 2a \right)
		=
		\xi_{0,a}(x) \left( \xi_{0,a}'(x) + 2a \right) > 0,
	\end{equation}
	and using \eqref{tmp.compare11}, we have for any $x\geq0$,
	$
	\xi_{0,a}'(x) \leq Q_{\infty}'(x),
	$
	which, after integration, implies for any $x\geq0$,
	\begin{equation*}
		\xi_{0,a}(x) \leq Q_{\infty}(x),
	\end{equation*}
	which can be combined with \eqref{tmp.compare11} to obtain $\xi_{0,a} = Q_{\infty}$.
	
	Hence by \eqref{SM.lemma}, \eqref{tmp:minkowski_infty} becomes
	\eqref{eq:minkowski1}, which completes the proof of the inequality in Theorem \ref{thm:minkowski1}. Using \eqref{eq:minkowski1}, the isoperimetric inequality \eqref{eq:isoperi1}, and the relation between profile functions \eqref{SM.16pi-4a}, we have
	\begin{equation}\label{eq:cor-minkowski}
		M(\Gamma) \geq \sqrt{16\pi \eta_{0,a}(V(\Omega)) - 4a \eta_{0,a}(V(\Omega))^2} = \xi_{0,a}(V(\Omega)),
	\end{equation}
	which proves the inequality in Corollary \ref{cor:minkowski1}.
	
	\textbf{Step 4: Equality case}
	
	We now discuss the case when equality holds in Theorem \ref{thm:minkowski1}. If equality in \eqref{eq:minkowski1} holds, following \cite{ghomi-spruck2023}, we will first prove $\Gamma$ is a geodesic sphere, then prove its enclosed domain is of constant curvature.
	We construct a new auxiliary function
	$$  \phi_{\infty}(t):= M_t^2 -  16\pi S(\Gamma_t) +2a S(\Gamma_t)^2 + 2a \eta_{0,a}(V(\Omega_t))^2. $$
	As above, using \eqref{eq:cor-minkowski}, one proves that $\phi_{\infty}$ is monotonically non-increasing along harmonic mean curvature flow. 
	
	Assume equality in \eqref{eq:minkowski1}, then $\phi_{\infty}(0)=0$. The monotonicity of \(\phi_\infty\), together with its endpoint behavior, forces \(\phi_\infty\) to be identically zero. Consequently every inequality used in deriving monotonicity is an equality. Equalities hold in \eqref{eq:CS}, which yields  $ M_t M'_t=-2\G_t^2$. The equality condition in Cauchy--Schwarz then implies that $G_t/H_t = \lambda(t)$ is constant on each time slice. Consequently, $\Gamma_t$ are parallel to $\Gamma$, thus each evolving surface is an equidistant hypersurface from the extinction point. The initial surface is therefore a geodesic sphere. 
	
	Moreover, the equality in \eqref{eq:GB} holds. This forces $\operatorname{Ric}(\nu_t)\equiv 2a$, which in turn yields that the sectional curvatures with respect to planes containing $\nu_t$ are equal to $a$, since they are no greater than $a$. Consequently, by the rigidity lemma \cite[Lem. 5.4]{ghomi-spruck2022}, all sectional curvatures of $N$ in the (geodesic) ball bounded by $\Gamma$ are equal to $a$, which implies that the domain bounded by $\Gamma$ is isometric to a geodesic ball in $\HH^3(a)$ and completes the proof.
\end{proof}
	
	Theorem \ref{thm:minkowski1} can be extended to general convex surfaces in Cartan-Hadamard 3-spaces by approximations using outer parallel surfaces. That is,
	\begin{thm}\label{thm:minkowski1-convex}
		Minkowski's inequality \eqref{eq:minkowski1} holds for all convex surfaces $\Gamma$ in a Cartan-Hadamard $3$-space $N$ with curvature $K\leq a\leq 0$.
	\end{thm}
	The proof follows from Ghomi-Spruck \cite[Section 3]{ghomi-spruck2023}.
	
	\section{Comparison of Inequality \eqref{eq:minkowski1} with \eqref{eq:minkowski3}} \label{sec:4}
		In this section, we compare inequality \eqref{eq:minkowski1} with inequality \eqref{eq:minkowski3} for convex surfaces in the standard hyperbolic space $\HH^3$, that is, $a=-1$, and show that inequality \eqref{eq:minkowski1} is sharper in this setting.
		
		\begin{prop}\label{prop:compare-inequality}
			Let $\Omega$ be a domain in $\HH^3$ such that its boundary $\Gamma$ is a smooth surface. Then
			\begin{align}
				\label{note2.prop}
				&\sqrt{16\pi S(\Gamma)+ 2 S(\Gamma)^2 + 2  \eta_{0,-1}(V(\Omega))^2} 
				\\
				\nonumber
				\geq
				&
				\sqrt{S(\Gamma)}\sqrt{S(\Gamma)+4\pi} + 4\pi \operatorname{arcsinh}\left( \sqrt{\frac{S(\Gamma)}{4\pi}} \right)
				+ 2 V(\Omega),
			\end{align}
			and equality holds only if $\Gamma$ is a sphere in $\HH^3$.
		\end{prop}
		\begin{proof}
			Define the two-variable functions
			\begin{align*}
				F_1(S,V)&:=\sqrt{16\pi S + 2S^2 + 2 \eta_{0,-1}(V)^2},
				\\
				F_2(S,V)&:=\sqrt{S}\sqrt{S+4\pi} + 4\pi \operatorname{arcsinh}\left( \sqrt{\frac{S}{4\pi}} \right) + 2V,
			\end{align*}
			where $S, V \geq 0$.
			We will show $F_1(S,V) \geq F_2(S,V)$ if $S \geq \eta_{0,-1}(V)$.
			
			We may compute
			\begin{align*}
				\p_{S}F_1(S,V) 
				&= \frac{16\pi + 4 S}{2\sqrt{16\pi S + 2S^2 + 2 \eta_{0,-1}(V)^2}}
				= \frac{8\pi + 2 S}{\sqrt{16\pi S + 2S^2 + 2 \eta_{0,-1}(V)^2}},
				\\
				\p_{S}F_2(S,V)
				&= \frac{2 S + 4\pi}{2\sqrt{S(S+4\pi)}}
				+
				4\pi \frac{1}{\sqrt{\frac{S}{4\pi}+1}} \frac{1}{\sqrt{4\pi}} \frac{1}{2 \sqrt{S}}
				= \frac{4\pi + S}{\sqrt{4\pi S + S^2}}.
			\end{align*}
			Therefore, for $S, V$ satisfying $S \geq \eta_{0,-1}(V)$, we have
			\begin{align}\label{Note2.tmp.compare}
				\p_{S}F_2(S,V) 
				= \frac{8\pi + 2S}{\sqrt{16\pi S + 4 S^2}}
				\leq \frac{8\pi + 2S}{\sqrt{16\pi S + 2S^2 + 2 \eta_{0,-1}(V)^2}} 
				= \p_{S}F_1(S,V),
			\end{align}
			hence
			\begin{align*}
				F_2(S,V) 
				= 
				&
				F_2(\eta_{0,-1}(V), V) + \int_{\eta_{0,-1}(V)}^{S} \p_{S}F_2(t,V) dt
				\\
				\leq
				&
				F_1(\eta_{0,-1}(V), V) + \int_{\eta_{0,-1}(V)}^{S} \p_{S}F_1(t,V) dt
				\\
				=
				&
				F_1(S,V).
			\end{align*}
			This is because $ F_2(\eta_{0,-1}(V), V) =F_1(\eta_{0,-1}(V), V)$. They both equal the total mean curvature of the sphere in $\HH^3$ with volume $V$.
			
			For a smooth closed surface \(\Gamma\) with enclosed domain $\Omega$, setting \(S=S(\Gamma)\) and \(V=V(\Omega)\), Theorem \ref{thm:isoperi1} gives $S \geq \eta_{0,-1}(V)$, and therefore the inequality follows.
			If equality holds, then the equality in \eqref{Note2.tmp.compare} holds, which implies $S = \eta_{0,-1}(V)$. By Theorem \ref{thm:isoperi1}, for a closed surface $\Gamma$ in $\HH^3$ with enclosed domain $\Omega$, if $S(\Gamma) = \eta_{0,-1}(V(\Omega))$, then $\Gamma$ is a sphere.
		\end{proof}
		
		Using notation in the proof of Proposition \ref{prop:compare-inequality}, if the ambient space $N$ is $\HH^3$, then \eqref{eq:minkowski1} can be written as follows: For any smooth convex surface $\Gamma$ in $\HH^3$, with enclosed domain $\Omega$,
		$
		M(\Gamma) \geq F_1(S(\Gamma), V(\Omega)).
		$
		Moreover, \eqref{eq:minkowski3} implies that for any smooth convex surface $\Gamma$ in $\HH^3$, with enclosed domain $\Omega$,
		$
		M(\Gamma) \geq F_2(S(\Gamma), V(\Omega)).
		$
		By Proposition \ref{prop:compare-inequality}, \eqref{eq:minkowski1} is sharper.

	\addtocontents{toc}{\protect\setcounter{tocdepth}{0}}%for hiding from table of contents
	\section*{Acknowledgments}
	I would like to thank my supervisor, Pengfei Guan, for introducing this problem to me and for many inspiring discussions. I would also like to thank Junfang Li for helpful discussions.
	
	\addtocontents{toc}{\protect\setcounter{tocdepth}{1}}%to resume listing of the sections  
	%\bibliography{references}

\end{document}